\documentclass[fleqn,11pt]{article}
\usepackage{amsmath,amssymb, amsthm}
\usepackage{mathrsfs,enumerate}
\allowdisplaybreaks

\numberwithin{equation}{section}
\newtheorem{theorem}{Theorem}[section]
\newtheorem{corollary}[theorem]{Corollary}

\newtheorem{lemma}[theorem]{Lemma}
\newtheorem{proposition}[theorem]{Proposition}

\theoremstyle{definition}
\newtheorem{definition}[theorem]{Definition}
\newtheorem{remark}[theorem]{Remark}

\newcommand{\bEm}{\mathbb{E}_m}
\newcommand{\diver}{{\rm div}}
\newcommand{\N}{\mathbb{N}}
\newcommand{\R}{\mathbb{R}}
\newcommand{\divg}{\diver_\gamma}
\newcommand{\Dg}{{\Delta_\gamma}}
\newcommand{\gga}{\gamma}
\newcommand{\Ldeu}{L^2(X,\gga)}

\newcommand{\FCb}{{\mathcal F C}_b^1}
\newcommand{\FCbt}{{\mathcal F C}_b^2}
\newcommand{\FCbk}{{\mathcal F C}_b^k}

\newcommand\scal[2]{{\left\langle #1 ,#2\right\rangle}}

\def\<{\langle}
\def\>{\rangle}

\def\1{{\bf 1}}

\title{A fractional isoperimetric problem in the Wiener space} 
\author{Matteo Novaga\footnote{Dipartimento di Matematica, 
Universit\`a di Pisa, Largo Bruno Pontecorvo 5, 56127 Pisa, Italy; e-mail: novaga@dm.unipi.it}, 
Diego Pallara\footnote{Dipartimento di Matematica e Fisica ``Ennio De Giorgi'', Universit\`a del Salento, 
P.O.B. 193, 73100, Lecce, Italy; e-mail: diego.pallara@unisalento.it} and 
Yannick Sire\footnote{Universit\'e Aix-Marseille, LATP, UMR CNRS 7353, Marseille, France; 
e-mail: sire@cmi.univ-mrs.fr}}
\date{}
\begin{document}
\maketitle
\begin{abstract}
We introduce a notion of fractional perimeter in an abstract Wiener space,
and we show that halfspaces are the only volume-constrained minimisers.
\end{abstract}


\section{Introduction}
The purpose of this paper is introducing a notion
of fractional perimeter in an abstract Wiener space, following 
the approach developed in the seminal work \cite{CRS},
and studying the symmetry properties of minimisers for this functional.
More precisely, our main result is to prove
that halfspaces are the unique isoperimetric sets 
for the fractional perimeter, 
as it happens for the usual perimeter (see \cite{CarKer01OnT}, 
\cite{Led98}, \cite[Remark 4.7]{AMMP}).
Owing to the well-known relation between 
the isoperimetric problem and the Allen-Cahn energy \cite{momo} (see also 
\cite{GoldNov} for an extension of the result to Wiener spaces,
and \cite{SaVa} for a nonlocal version in finite dimensions), 
we also prove the one-dimensional symmetry
of minimisers of the corresponding nonlocal Allen-Cahn energy 
(see Theorem \ref{th.Doublewell}).
We now state the main result of this paper
(see Section \ref{secnot} for the precise definitions).

\begin{theorem}\label{thisoperimetric}
For any $s\in (0,1)$ and $m\in (0,1)$ there exists 
a set $E_m\subset X$ which solves the isoperimetric problem
\begin{equation}\label{eqisoper}
\min \Bigl\{ P_{\gamma,s}(E):\ E\subset X,\ \gamma(E)=m\Bigr\}. 
\end{equation}
Moreover, the set $E_m$ is necessarily a half-space, i.e., 
$E=\{\hat{h}<c\}$ for some $h\in H$ and $c\in{\mathbb R}$. 
\end{theorem}
The proof of Theorem \ref{thisoperimetric} is based on the extension technique introduced in \cite{CS}. 
Indeed, the fractional perimeter, and more generally the fractional Sobolev seminorm defined in 
\eqref{defHdot}, can be obtained via the minimisation of a Dirichlet energy after adding an extra 
variable that lies on a half-line endowed with a degenerate measure. As a consequence, the 
isoperimetric problem can be tackled by studying this minimisation problem. To this aim, 
we split the Dirichlet functional in two contributions $J_1$ and $J_2$ in a natural way 
and show that both are decreasing under Ehrhard symmetrisation defined in \eqref{defI*}. 
These results are proved in Lemmas \ref{brock} (which seems to be new
even in the finite dimensional case) and \ref{ehrhard}, respectively. In the first proof we adapt 
a technique in \cite{Brock}, in the second we use cylindrical approximations to extend 
the result from the finite dimensional to the infinite dimensional setting, see Lemma \ref{lemGN}. 
These symmetrisation results we believe to be interesting on their own.

Partial symmetrisations in product spaces are also used in \cite{FusMagPra}, with the aim of studying 
isoperimetric problems with respect to product measures. Also in this case, it is shown that 
the advantage of symmetrising with respect to a set of variables is not affected by the others. 

\section{Notation and preliminary definitions}\label{secnot}
We collect here the definitions and the preliminaries results needed in the sequel. 
The first two subsections are devoted to the structure of the Wiener space; for all this 
results we refer to the book \cite{B}. In the third subsection we introduce the fractional 
perimeters and Sobolev seminorms and use the extension technique in \cite{CS,TS} to show
some further preliminary results. 

\subsection{The Wiener space}\label{secwiener}

An abstract Wiener space is a triple $(X,\gamma,H)$ where
$X$ is a separable Banach space, endowed with the norm $\|\cdot\|_X$,
$\gamma$ is a nondegenerate centred Gaussian measure,
and $H$ is the Cameron-Martin space associated with the measure $\gamma$, that is,
$H$ is a separable Hilbert space densely embedded in $X$, endowed with the inner product
$[\cdot, \cdot ]_H$ and with the norm $|\cdot |_H$. The
requirement that $\gamma$ is a centred Gaussian measure means that
for any $x^*\in X^*$, the measure $x^*_\#\gamma$ is a centred Gaussian
measure on the real line $\R$, that is, the Fourier transform
of $\gamma$ is given by
\[
\hat \gamma(x^*) = \int_X 
e^{-i\scal{x}{x^*}}\, d\gamma (x)=\exp\left(
-\frac{\scal{Qx^*}{x^*}}{2}
\right),\qquad \forall x^*\in X^*;
\]
here the operator $Q\in {\mathcal L}(X^*,X)$ is the covariance operator and it is
uniquely determined by the formula
\[
\scal{Qx^*}{y^*}=\int_X \scal{x}{x^*}\scal{x}{y^*}d\gamma(x),\qquad \forall x^*,y^*\in X^*.
\]
The nondegeneracy of $\gamma$ implies that $Q$ is positive definite: the boundedness
of $Q$ follows by Fernique's Theorem (see \cite[Theorem 2.8.5]{B}), 
asserting that there exists a positive number $\beta>0$
such that
\[
 \int_X e^{\beta\|x\|_X^2}d\gamma(x)<+\infty.
\]
This implies also that the maps $x\mapsto \scal{x}{x^*}$ belong to 
$L^p_\gamma(X)$ for any $x^*\in X^*$ and $p\in [1,+\infty)$, where $L^p_\gamma(X)$ 
denotes the space of all $\gamma$-measurable functions $f:X\to \R$
such that 
\[
\int_X |f(x)|^p d\gamma(x)<+\infty.
\]
In particular, any element $x^*\in X^*$ can be seen as a map $x^*\in L^2_\gamma(X)$, and
we denote by $R^*: X^*\to {\mathcal H}$ the identification map $R^*x^*(x):=\scal{x}{x^*}$.
The space ${\mathcal H}$ given by the closure of $R^*X^*$ in $L^2_\gamma(X)$
is usually called reproducing kernel. 
By considering the map $R: {\mathcal H}\to X$
defined through the Bochner integral 
\[
R\hat{h} := \int_X \hat{h}(x)x\, d\gamma(x),
\]
we obtain that $R$ is an injective $\gamma$--Radonifying operator, which is
Hilbert--Schmidt when $X$ is Hilbert. We also have  
$Q=RR^*:X^*\to X$. The space $H:=R{\mathcal H}$, equipped with the inner product $[\cdot,\cdot]_H$
and norm $|\cdot|_H$ induced by ${\mathcal H}$ via $R$, is the Cameron-Martin space  
and is a dense subspace of $X$. The continuity of $R$ implies
that the embedding of $H$ in $X$ is continuous, that is, there exists $c>0$
such that
\[
 \|h\|_X \leq c|h|_H,\qquad \forall h\in H.
\] 
We have also that the measure $\gamma$ is absolutely continuous with respect
to translation along Cameron-Martin directions; in fact, for $h\in H$, $h=Qx^*$, 
the measure
$\gamma_h(B)=\gamma(B-h)$ is absolutely continuous with respect to $\gamma$ with
density given by
\[
d\gamma_h(x)=\exp\left(
\scal{x}{x^*}-\frac{1}{2}|h|_H^2
\right)d\gamma(x) .
\]

\subsection{Cylindrical functions and differential operators}\label{cylindrical}

For $j\in \N$ we choose $x^*_j\in X^*$ in such a way that 
$\hat h_j:= R^*x_j^*$, or equivalently $h_j:=R\hat h_j=Qx^*_j$, form an orthonormal basis
of $H$. 
We order the vectors $x^*_j$ in such a way that the numbers $\lambda_j:=\|x^*_j\|_{X^*}^{-2}$
form a non-increasing sequence.
Given $m\in\mathbb N$, we also let $H_m:=\langle h_1,\ldots, h_m\rangle\subseteq H$, 
and $\Pi_m: X\to H_m$ be the closure of the orthogonal projection from $H$ to $H_m$
\[
\Pi_m(x) := \sum_{j=1}^m \scal{x}{x^*_j}\, h_j \qquad x\in X.
\]
The map $\Pi_m$ induces the decomposition $X\simeq H_m\oplus X_m^\perp$, with $X_m^\perp:= {\rm ker}(\Pi_m)$,
and $\gamma=\gamma_m\otimes\gamma_m^\perp$,
with $\gamma_m$ and $\gamma_m^\perp$ Gaussian measures on $H_m$ and $X_m^\perp$ respectively, 
having $H_m$ and $H_m^\perp$ as Cameron-Martin spaces. 
When no confusion is possible we identify $H_m$ with $\R^m$;
with this identification the measure
$\gamma_m={\Pi_m}_\#\gamma$ is the standard Gaussian measure on $\R^m$.
Given $x\in X$,
we denote by $\underline x_m\in H_m$ the projection $\Pi_m(x)$, 
and by $\overline x_m\in X_m^\perp$ the infinite dimensional component of $x$, so that 
$x=\underline x_m+\overline x_m$.
When we identify $H_m$ with $\R^m$ we rather write 
$x=(\underline x_m,\overline x_m)\in \R^m\oplus X_m^\perp$.

We say that $u:X\to \R$ is a {\em cylindrical function} if $u(x)=v(\Pi_m (x))$ for some $m\in\mathbb N$ and 
$v:\R^m\to \R$. 
We denote by $\FCbk(X)$, $k\in\mathbb N$, 
the space of all $C^k_b$ cylindrical functions, that is, functions of the form $v(\Pi_m (x))$
with $v\in C_b^k(\R^m)$, with continuous and bounded derivatives up to the order $k$. 
We denote by $\FCbk(X,H)$ the space generated by all functions of the form 
$u h$, with $u\in \FCbk(X)$ and $h\in H$.

Given $u\in \Ldeu$, we consider the canonical cylindrical approximation operators $\bEm$ given by
\begin{equation}\label{cancylapprox}
\bEm u (x)=\int_{X_m^\perp} u(\Pi_m(x),y) \,d\gga_m^\perp(y).
\end{equation}
Notice that $\bEm u$ depends only on the first $m$ variables 
and $\bEm u$ converges  to $u$ in $L^p_\gamma(X)$ for all $1\leq p<\infty$.
We let
\[
\begin{array}{ll}
\displaystyle{\nabla_\gamma u := \sum_{j\in\mathbb N}\partial_j u\, h_j} & {\rm for\ } u\in \FCb(X)
\\
\\
\displaystyle{\divg \varphi := \sum_{j\geq 1}\partial^*_j [\varphi,h_j]_H} & {\rm for\ }\varphi\in \FCb(X,H)
\\
\\
\displaystyle{\Dg u := \divg\nabla_\gamma u} & {\rm for\ } u\in \FCbt(X)
\end{array}
\]
where $\partial_j := \partial_{h_j}$ and
$\partial_j^* := \partial_j - \hat h_j$ is the adjoint operator of $\partial_j$.
With this notation, the following integration by parts formula holds:
\begin{equation}\label{inp}
\int_X u\, \divg \varphi\,d\gamma = -\int_X [\nabla_\gamma u,\varphi]_H\, d\gamma
\qquad \forall \varphi\in \FCb(X,H).
\end{equation}
In particular, thanks to \eqref{inp}, the operator $\nabla_\gamma$ is closable in $L^p_\gamma(X)$,
and we denote by $W^{1,p}_\gamma(X)$ the domain of its closure. The Sobolev spaces 
$W^{k,p}_\gamma(X)$, with $k\in\mathbb N$ and $p\in [1,+\infty]$, can be defined analogously,
and $\FCbk(X)$ is dense in $W^{j,p}_\gamma(X)$, for all $p<+\infty$ and $k,j\in\mathbb N$ with $k\ge j$.
Given a vector field $\varphi \in L^{p}_\gamma(X;H)$, $p\in (1,\infty]$, using \eqref{inp} we can define
$\mathrm{div}_\gamma \, \varphi$ in the distributional sense,
taking test functions $u$ in  $W^{1,q}_\gamma(X)$ with
$\frac{1}{p}+\frac{1}{q} = 1$. 
We say that
$\mathrm{div}_\gamma\, \varphi \in L^p_\gamma(X)$ if this linear functional can be extended to all test 
functions $u\in L^{q}_\gamma(X)$. This is true in particular if $\varphi\in W^{1,p}_\gamma(X;H)$.

Let $u\in W^{2,2}_\gamma(X)$, $\psi\in \FCb(X)$ and $i,j\in \mathbb N$.
{}From \eqref{inp}, with $u=\partial_j u$ and $\varphi=\psi h_i$, we get
\begin{equation}\label{parts}
\int_X \partial_j u\,\partial_{i}\psi \,d\gamma = 
\int_X -\partial_i(\partial_{j}u)\,\psi+ \partial_ju\,\psi\langle x^*_i,x\rangle d\gamma
\end{equation}
Let now $\varphi\in \FCb(X,H)$. If we apply \eqref{parts} with $\psi=[\varphi,h_j]=:\varphi^j$, we obtain 
\[
\int_X \partial_j u\,\partial_{i}\varphi^j \,d\gamma = 
\int_X -\partial_j(\partial_{i}u)\,\varphi^j
+ \partial_ju\,\varphi^j\langle x^*_i,x\rangle d\gamma
\]
which, summing up in $j$, gives 
\[
\int_X [\nabla_\gamma u,\partial_i \varphi]\,d\gamma = \int_X -[\nabla_\gamma (\partial _i u), \varphi] 
+  [\nabla_\gamma u,\varphi] \langle x^*_i,x\rangle d\gamma
\]
for all $\varphi\in \FCb(X,H)$.

The operator $\Dg:W^{2,p}_\gamma(X)\to L^p_\gamma(X)$ is usually called the 
Ornstein-Uhlenbeck operator on $X$.
Notice that, if $u$ is a cylindrical function, that is $u(x)=v(y)$ with  
$y=\Pi_m(x)\in\R^m$ and $m\in\mathbb N$,
then
\[
\Dg u = \sum_{j=1}^m \partial_{jj}u-\langle x_j^*,x\rangle\partial_{j}u = 
\Delta v - \langle y,\nabla v \rangle_{\R^m}\,.
\]

\subsection{Fractional Sobolev spaces and fractional perimeters}

Since the operator $-\Delta_\gamma$ is positive and self-adjoint in 
$L^2_\gamma(X)$, one can define its fractional powers 
by means of the standard formula in spectral theory  
\[
\displaystyle(-\Delta_{\gamma})^{s}=\frac1{\Gamma(-s)}\int_0^\infty
\left(e^{t\Delta_{\gamma}}-\mbox{Id}\right)\frac{dt}{t^{1+s}},
\]
where $s\in (0,1)$
and $e^{t\Delta_{\gamma}}$ denotes the Ornstein-Uhlenbeck semigroup on $X$.

For non local PDEs involving the fractional laplacian it is by now classical 
to use the so-called Caffarelli-Silvestre extension (see \cite{CS}). 
Here we use a general formulation of it, 
due to Stinga and Torrea \cite{TS}, which can easily adapted 
to our infinite dimensional setting. 
More precisely, a consequence of their main result is the following: 

\begin{theorem}\label{extension}
Let $u \in {\rm dom}((-\Delta_\gamma)^s)$. A solution of the extension problem 
\begin{equation}\label{pbextension}
\left \{ \begin{array}{ll}
\displaystyle{\Delta_\gamma v+ \frac{1-2s}{y}\partial_yv + \partial^2_yv=0}\qquad &\mbox{on } \, X\times (0, +\infty)
\\ \\ 
v(x,0)=u &\mbox{on } \, X, \\
\end{array} \right . 
\end{equation}
is given by 
\[
v(x,y)=\frac{1}{\Gamma(s)}\int_0^\infty e^{t\Delta_\gamma} ((-\Delta_\gamma)^su)(x)e^{-y^2/4t}\frac{dt}{t^{1-s}}
\]
and furthermore, one has in $L^2_\gamma(X)$ 
\begin{equation}\label{eqcs}
-\lim_{y\to 0^+} y^{1-2s}\partial_yv(x,y)=
\frac{2s\Gamma(-s)}{4^s\Gamma(s)}(-\Delta_\gamma)^s u(x). 
\end{equation}
\end{theorem}

After defining the fractional laplacian, let us introduce the fractional Sobolev space 
\[
H_\gamma^s(X)=\Bigl\{u\in L^2_\gamma(X):\ [u]_{H_\gamma^s}<\infty\Bigr\} 
\]
where
\begin{equation}\label{defHdot}
[u]_{H_\gamma^s}^2 = 
\inf\Bigl\{\int_{X\times{\mathbb R}^+}
\left(|\nabla_{\gamma} v|_H^2+ |\partial_y v|^2\right)
y^{1-2s}d\gamma(x)dy:\ 
v\in H^1_{\rm loc}(X\times{\mathbb R}^+),\ v(\cdot, 0)=u(\cdot)\Bigr\}.
\end{equation}
The space $H_\gamma^s$ is endowed with the Hilbert norm 
\[
\|u\|^2_{H_\gamma^s}=\|u\|^2_{L^2_\gamma}+[u]_{H_\gamma^s}^2.
\] 
\begin{remark}\rm
Let us define the space 
\[
H^1(X\times{\mathbb R}^+\!,\gamma\otimes y^{1-2s}dy)=\Bigl\{ 
v\in H^1_{\rm loc}(X\times{\mathbb R}^+): \!
\int_{X\times{\mathbb R}^+}\!\!\!\!
\left(|v|^2+|\nabla_{\gamma} v|_H^2+ |\partial_y v|^2\right)
y^{1-2s}d\gamma(x)dy <\infty\Bigr\}.
\]
A function $u\in L^2_\gamma(X)$ belongs to $H_\gamma^s$ if and only if there is 
$v_u\in H^1(X\times{\mathbb R}^+,\gamma\otimes y^{1-2s}dy)$ such that the 
infimum in \eqref{defHdot} is attained by $v_u$. We may therefore define the inner product
\[
\langle u,w\rangle_{\dot{H}_\gamma^s}=\int_{X\times{\mathbb R}^+}
\big( [\nabla_{\gamma} v_u,\nabla_{\gamma} v_w]_H
+ \partial_y v_u \,\partial_y v_w
\big) y^{1-2s}d\gamma(x)dy,\qquad u,w\in H_\gamma^s(X).
\]
\end{remark}

We relate in the next lemma the fractional laplacian 
with the spaces described above.
\begin{lemma}\label{Euler}
for every $u,w\in H_\gamma^s$ with $u\in {\rm dom} ((-\Delta_\gamma)^s)$ the following equality holds:
\[
\langle u,w\rangle_{\dot{H}_\gamma^s}= c_s \int_X (-\Delta_\gamma)^su\,w\ d\gamma,
\]
where $c_s$ is the constant in \eqref{eqcs}.
\end{lemma}
\begin{proof}
For $u,w\in H_\gamma^s(X)$, let $v_u$ be as above. It easily follows from the minimality
and elliptic regularity that $v_u$ is a solution of problem \eqref{pbextension}. Indeed, let us 
consider the test function 
$\varphi(x)\psi(y)$ with $\varphi\in {\mathcal F}C^\infty_b(X)$ and 
$\psi\in C_c^\infty({\mathbb R})$; we have 
\begin{align*}
\langle u,\varphi\psi\rangle_{\dot{H}_\gamma^s}=& 
\int_{X\times{\mathbb R}^+} 
\Bigl[ [\nabla_{\gamma} v_u,\nabla_\gamma \varphi(x)]_H\psi(y) 
+\varphi(x)\partial_yv_u\psi'(y)\Bigr] y^{1-2s}d\gamma(x)dy
\\
=&\int_{X\times{\mathbb R}^+} 
\bigl(-\Delta_\gamma v_u-\partial_y^2v_u -\frac{1-2s}{y}\partial_yv_u\bigr)
\varphi(x)\psi(y) y^{1-2s}d\gamma(x)dy
\\
&- \int_X\lim_{y\to 0+}(y^{1-2s} \partial_yv_u(x,y)\psi(y))\varphi(x)\, d\gamma(x).
\end{align*}
Since $v_u(\cdot,0)=u(\cdot)$, from \eqref{eqcs} and the density 
of the test functions in $H^s_\gamma$ 
we obtain the thesis. 
\end{proof}

We are now ready to define the fractional perimeter of a set in $X$. 

\begin{definition}
For every measurable set $E\subset X$ ans $0<s<1$ we define the fractional $s$-perimeter by setting 
\[
P_{\gamma,s}(E)=[\chi_E]_{H_\gamma^s}
\]
according to \eqref{defHdot}. We say that $E$ has finite $s$-perimeter in 
$X$ if $P_{\gamma,s}(E)<\infty$.
\end{definition}

Let us show that a form of the coarea formula holds 
in this framework as well (see \cite{Visintin}). 

\begin{proposition}\label{prop.coarea}
Setting for $u\in L^1_\gamma(X)$
\[
V_s(u) = \int_{\mathbb R}P_{\gamma,s}(\{u>t\})\, dt, 
\]
$V_s$ is convex and lower semicontinuous on $L^1_\gamma(X)$. Moreover, 
if $u_n={\mathbb E}_n[u]$ are the canonical cylindrical approximation of $u$ then 
$V_s(u)\leq V_s(u_n)$.  
\end{proposition}
\begin{proof}
The convexity of $V_s$ has been proved in 
\cite[Proposition 3.4]{ChaGiaLus}, while the 
lower semicontinuity easily follows from the lower semicontinuity 
of perimeters. The last 
inequality follows immediately from Jensen's inequality. 
\end{proof}

\section{The fractional isoperimetric problem}

In order to discuss the isoperimetric properties of half-spaces, 
following \cite{EhrScand} we introduce a suitable 
notion of symmetrisation. 
For $h\in H$ with $|h|_H=1$, we consider the projection
$x'=\pi_hx=x-\hat{h}(x)h$ and write $x=x'+th$ with $t\in{\mathbb R}$. 
Therefore, for fixed $h\in H$ and for any $I\subset{\mathbb R}$ we set 
\begin{equation}\label{defI*}
I^*=(-\infty, \phi^{-1}(\gamma_1(I)), \qquad\text{where}\quad \phi(t)=\int_{-\infty}^t e^{-s^2/2}ds.
\end{equation}
In the same vein, for every measurable function $u:X\to{\mathbb R}$ we define the symmetrised function 
\begin{equation}\label{defu*}
u_h^*(x'+th)=\sup\Bigl\{c\in {\mathbb R}:\ t\in\{u(x',\cdot)>c\}^*\Bigr\}. 
\end{equation}
Since symmetrisation preserves characteristic functions, we may define the set 
$E_h^*$ through the equality 
\[
\chi_{E_h^*}=(\chi_E)_h^* .
\]
The proof of Theorem \ref{thisoperimetric} relies on the following lemma.

\begin{lemma}\label{brock}
Let $v\in H^1(X\times{\mathbb R}^+,\gamma\otimes y^{1-2s}dy)$ and let $h\in X^*$ with $|h|_H=1$. 
Let $v^*_h$ be as in \eqref{defu*} and let
\[
J_1(v):=\int_{X\times {\mathbb R}^+} |\partial_yv|^2\, y^{1-2s}d\gamma(x)dy.
\]
Then we have the inequality $J_1(v_h^*) \leq J_1(v)$.  
\end{lemma}
\begin{proof} The proof follows that of Theorem 1 in \cite{Brock} with minor modifications, 
we repeat it for the reader's convenience. There are some differences: Brock's result is in finite 
dimensions, the underlying measure is the Lebesgue one and he uses the Steiner symmetrisation, 
whereas we work in $X\times{\mathbb R}^+$ with the product measure $\gamma\otimes y^{1-2s}dy$
and we are concerned with the Ehrhard symmetrisation. On the other hand, the functionals 
considered by Brock are much more general than ours. \\
In order to simplify the notation, suppose that $h=h_1$, write as before $x=x'+th_1$ split 
$X=H_1\oplus X_1^\perp$ and decompose the gaussian measure as $\gamma=\gamma_1\otimes\gamma_1^\perp$. 
Since for every $v\in H^1(X\times{\mathbb R}^+,\gamma\otimes y^{1-2s}dy)$ we have
\[
J_1(v) = \int_{X_1^\perp}\Bigl(\int_{\mathbb R}\int_{{\mathbb R}^+}
|\partial_yv(x',t,y)|^2y^{1-2s}d\gamma_1(t)dy\Bigr)d\gamma_1^\perp(x'),
\]
we may limit ourselves to the inner double integral, for fixed $x'$. Moreover, 
following the reduction explained in \cite{Brock}, we may deal only with the dense class of {\em nice functions}, i.e., 
piecewise affine functions $v:\R\times\R^+\to\R$ such that 
for every $c>\inf v$ the equation $v(t,y)=c$ has for every $y\in\R^+$ a finite (even) 
number of solutions $t_1,\ldots,t_{2m}$. Once the result is proved for nice functions, the 
general case follows as in \cite{Brock}. For $v$ nice, set $\Omega=\{v>0\}$ and 
decompose the vertical set 
$\{(y,z)\in\R^+\times\R^+:\ \exists\, (t,y)\in\Omega\text { such that } v(t,y)=z\}$ 
into $N$ disjoint domains $G_j$ such that for any $(y,z)\in G_j$ the equation $v(t,y)=z$ 
has exactly $2m$ (with $m$ depending on $j$) solutions $t=t^j_k,\ k=1,\ldots,2m(j)$. 
Thus $v$ can be represented in each $G_j$ by the inverse functions $t=t^j_k(y,v)$. In each 
domain $G_j$ the following identities hold:
\begin{align*}
\partial_t v(t^j_k,y)&=
\Bigl(\frac{\partial t^j_k}{\partial v}\Bigr)^{-1}
\left\{\begin{array}{ll}
        >0\quad &\text{ if $k$ is odd}
        \\
        <0 &\text{ if $k$ is even}
       \end{array}
\right.\\
\partial_y v(t^j_k,y)&= 
-\frac{\partial t^j_k}{\partial y}\Bigl(\frac{\partial t^j_k}{\partial v}\Bigr)^{-1}.
\end{align*}
Since $v$ is nice, all the derivatives of $t^j_k$ are constant in $G_j$ and therefore 
the rearranged function $v^*$ is nice, too. Moreover the symmetrisation procedure 
reduces the solutions of the equation $v^*(t,y)=z$ to only one, i.e., the following 
\[
T^j= \phi^{-1}\Bigl(\sum_{k=1}^{2m(j)}(-1)^{k-1} \phi(t^j_k) \Bigr)
\]
(where $\phi$ is introduced in \eqref{defI*}) in each $G_j$, for every $y\in\R^+$. 
Differentiating we get 
\begin{align*}
 \gamma_1(T^j)\frac{\partial T^j}{\partial y} &= 
 \sum_{k=1}^{2m(j)}(-1)^{k-1} \gamma (t^j_k) \frac{\partial t^j_k}{\partial y}
\\
 \gamma_1(T^j)\frac{\partial T^j}{\partial z} &= 
 \sum_{k=1}^{2m(j)} \gamma (t^j_k) \Bigl|\frac{\partial t^j_k}{\partial y}\Bigr| .
\end{align*}
It follows (with $x'\in X_1^\perp$ fixed) 
\begin{align*}
 \int_{\mathbb R}\int_{{\mathbb R}^+}|\partial_yv(x',t,y)|^2y^{1-2s}\gamma_1(t)dydt
 &= \sum_{j=1}^N \int_{G_j}\sum_{k=1}^{m(j)}
 \Bigl|\frac{\partial t^j_k}{\partial y}\Bigr|^2\, 
 \Bigl|\frac{\partial t^j_k}{\partial z}\Bigr|^{-1}\gamma_1(t^j_k)dydz
 \\
 \int_{\mathbb R}\int_{{\mathbb R}^+}|\partial_yv^*(x',t,y)|^2y^{1-2s}\gamma_1(t)dydt
 &= \sum_{j=1}^N \int_{G_j}
 \Bigl|\frac{\partial T^j}{\partial y}\Bigr|^2\, 
 \Bigl|\frac{\partial T^j}{\partial z}\Bigr|^{-1}\gamma_1(T^j)dydz
 \\
 &= \sum_{j=1}^N \int_{G_j}
 \frac{\left|\sum_{k=1}^{2m(j)}(-1)^{k-1} \gamma (t^j_k) 
\dfrac{\partial t^j_k}{\partial y}\right|^2} 
 {\left|\,\sum_{k=1}^{2m(j)} \gamma (t^j_k) 
\Bigl|\frac{\partial t^j_k}{\partial y}\Bigr|\, \right|}dydz.
\end{align*}
Setting
\[
 c^j_k = \gamma_1(t^j_k)\frac{\partial t^j_k}{\partial y},
 \qquad 
 b^j_k = \gamma_1(t^j_k)\Bigl|\frac{\partial t^j_k}{\partial z}\Bigr|,
\]
we have the following equivalence:
\begin{align*}
 &\int_{\mathbb R}\int_{{\mathbb R}^+}|\partial_yv(x',t,y)|^2y^{1-2s}\gamma_1(t)dydt
 \geq \int_{\mathbb R}\int_{{\mathbb R}^+}|\partial_yv^*(x',t,y)|^2y^{1-2s}\gamma_1(t)dydt
  \\ 
 &\Longleftrightarrow\qquad\sum_{k=1}^{2m(j)}\frac{(c_k^j)^2}{b^j_k} \geq 
 \Bigl(\sum_{k=1}^{2m(j)}(-1)^{k-1}c^j_k\Bigr)^2\Bigl|\sum_{j=1}^{2m(j)}b^j_k\Bigr|^{-1}
 \ \ \forall\, j=1,\ldots,N.
\end{align*}
But, the last inequality is nothing but the Cauchy-Schwarz inequality:
\[
 \Bigl(\sum_{k=1}^{2m(j)} (-1)^{k-1}c_k\Bigr)^2 = 
 \Bigl(\sum_{k=1}^{2m(j)} (-1)^{k-1}\frac{c_k}{\sqrt{b_k}} \sqrt{b_k}\Bigr)^2
 \leq \sum_{k=1}^{2m(j)} \frac{c_k^2}{b_k} \sum_{k=1}^{2m(j)} b_k
\]
and the thesis follows.
\end{proof}

Let us show that the $L^2_\gamma(X)$ norm of the gradient is also decreasing under Ehrhard rearrangement. 

\begin{lemma}\label{lemGN}
Let $u\in H^{1}_\gamma(X)$, and let $h\in X^*$ with $|h|_H=1$.
Then $u^*_h\in H^{1}_\gamma(X)$ and
\begin{equation}\label{polya}
\int_{X}|\nabla_\gamma u_h^*|_H^2 d\gamma \le
\int_{X}|\nabla_\gamma u|_H^2 d\gamma\,.
\end{equation}
\end{lemma}
\begin{proof}
In \cite[Th. 3.1]{EhrAnnENS} the inequality \eqref{polya} is proven for Lipschitz
functions in finite dimensions. We extend it by approximation to Sobolev functions in $H^1_\gamma(X)$.

We let $u_n \in \FCb(X)$ be the canonical cylindrical approximation of $u$
defined in \eqref{cancylapprox}.
Since $u_n\to u$ in $H^1_\gamma(X)$, we have
$(u_n)^*_h\to u^*_h$ in $L^2_\gamma(X)$, so that by the lower semicontinuity
of the $H^1_\gamma$ norm we obtain
\[
\int_{X}|\nabla_\gamma u^*_h|_H^2 d\gamma \leq
\liminf_{n\to \infty} \int_{X}|\nabla_\gamma (u_n)^*_h|_H^2 d\gamma
\leq \liminf_{n\to \infty}
\int_{X}|\nabla_\gamma u_n|_H^2 d\gamma
=\int_{X}|\nabla_\gamma u|_H^2 d\gamma\,.
\]
\end{proof}

{}From Lemma \ref{lemGN} we immediately get the following result.

\begin{lemma}\label{ehrhard}
Let $v\in H^1(X\times{\mathbb R}^+,\gamma\otimes y^{1-2s}dy)$ and let $h\in
X^*$ with $|h|_H=1$.
Letting $v^*_h$ be as in \eqref{defu*} and
\[
J_2(v)=\int_{X\times {\mathbb R}^+} |\nabla_{\gamma} v|_H^2\,
y^{1-2s}d\gamma(x)dy,
\]
we have the inequality $J_2(v^*_h) \leq J_2(v)$.
\end{lemma}

{}From \eqref{defHdot}, Lemma \ref{brock}
and Lemma \ref{ehrhard} we immediately get
the following result:

\begin{corollary}\label{isoperimetry}
 If $u\in H_\gamma^s(X)$ then for every $h\in H$ we have $u_h^*\in H_\gamma^s(X)$ and 
\[
[u_h^*]_{H_\gamma^s} \leq [u]_{H_\gamma^s}\,.
\] 
\end{corollary}

Given $u\in L^2_\gamma(X)$, let $S_u:{\mathbb R}\to{\mathbb R}$ be the decreasing function 
defined through its inverse by the equality
\[
S_u^{-1}(t) = \phi^{-1}(\gamma(\{u>t\}), 
\]
for $\phi$ as in \eqref{defI*}, so that $\gamma(\{u>t\}) = \gamma(\{S_u>t\})$. 

\begin{theorem}\label{iteration}
Let $u\in H_\gamma^s$. Then 
\begin{equation}\label{ineq}
[S_u]_{H_{\gamma_1}^s} \leq [u]_{H_\gamma^s}\,,
\end{equation}
with equality if and only if $u$ is one-dimensional, that is,
$u(x)=S_u(\hat h(x))$ for some $h\in H$ with $|h|=1$.
\end{theorem}
\begin{proof}
We first show the inequality \eqref{ineq}.
Let $(u_n)$ be the canonical cylindrical approximation of $u$ defined in \eqref{cancylapprox}, let $(h_k)$ be a sequence 
dense in $\{h\in H_n: |h|_H=1\}$ and let $u_{n,k}$ be iteratively defined by 
$u_{n,0}=u_n$ and $u_{n,k}=(u_{n,k-1})^*_{h_k}$ as in \eqref{defu*}. Then, 
$\|u_{n,k}\|_{L^2_\gamma(X)} = \|u_{n}\|_{L^2_\gamma(X)}$ for every $k$ and by the preceding lemmas, 
we have that $[u_{n,k}]_{H_\gamma^s}\leq [u_{n}]_{H_\gamma^s}$, hence (up to a subsequence 
that we don't relabel) the sequence $(u_{n,k})$ converges to a function $\tilde{u}_n$ in $L^2_\gamma(X)$ 
with $[\tilde{u}_n]_{H_\gamma^s} \leq [u_n]_{H_\gamma^s}$. 
Since $\tilde{u}_n$ is symmetric 
with respect to all the directions in $H_n$, it can be written as $\tilde{u}_n(x)=S_{u_n}(\hat{h}(x))$
for some $h\in H_n$. From Lemma \ref{brock}
and Lemma \ref{ehrhard} it follows that 
\[
[S_{u_n}]_{H_{\gamma_1}^s} = [S_{u_n}\circ \hat h]_{H_\gamma^s}
\leq [u_n]_{H_\gamma^s} \leq [u]_{H_\gamma^s} .
\]
Passing to the limit as $n\to\infty$ and noting that $S_{u_n}\to S_u$ in $L^2_{\gamma_1}(\R)$,
we get the inequality \eqref{ineq}.  

\smallskip

Assume now that the equality holds in \eqref{ineq}.
Again by Lemma \ref{brock}
and Lemma \ref{ehrhard}, this implies that 
\[
\int_{X\times \R^+}|\partial_y v_{S_u}|^2y^{1-2s}\,\gamma(x)\,dy\, =\,
\int_{X\times \R^+}|\partial_y v_{u}|^2y^{1-2s}\,\gamma(x)\,dy
\]
and
\[
\int_{\R^+}\|\nabla_{\gamma} v_{S_u}(\cdot,t)\|^2_{L^2_\gamma(X)}y^{1-2s}\,dy\, =\,
\int_{\R^+}\|\nabla_{\gamma} v_{u}(\cdot,t)\|^2_{L^2_\gamma(X)}y^{1-2s}\,dy\,,
\]
where $v_{S_u}, v_u$ are the corresponding minimisers 
of the right-hand side of \eqref{defHdot}. Hence, 
for a.e. $t>0$ we have
\[
\|\nabla_{\gamma} v_{S_u}(\cdot,t)\|_{L^2_\gamma(X)} 
= 
\|\nabla_{\gamma} v_u(\cdot,t)\|_{L^2_\gamma(X)}\,.
\]
Thanks to \cite[Prop. 3.12]{GoldNov},
it follows that $v_u$ is one-dimensional for a.e. $t>0$,
which implies that $u$ is also one-dimensional,
and concludes the proof.
\end{proof}

A direct consequence of Theorem \ref{iteration} is the following 
symmetry result:

\begin{theorem}\label{th.Doublewell}
Let $m>0$ and $F:{\mathbb R}\to{\mathbb R}$ be lower semicontinuous, and assume that the problem 
\begin{equation}\label{1Dproblem}
\min\Bigl\{ [w]_{H_{\gamma_1}^s} + \int_{\mathbb R} F(w)\, d\gamma_1 :\ 
\int_{\mathbb R}w\, d\gamma_1=m\Bigr\}
\end{equation}
admits a minimiser. Then the unique minimisers of the problem 
\begin{equation}\label{Xproblem}
\min\Bigl\{ [u]_{H_\gamma^s} + \int_{X} F(u)\, d\gamma :\ 
\int_{X}u\, d\gamma=m\Bigr\}
\end{equation}
are given by $u(x)=\varphi(\hat{h}(x))$ for some minimiser $\varphi$ of problem \eqref{1Dproblem} 
and for some $h\in H$. 
\end{theorem}

We can conclude with the proof of Theorem \ref{thisoperimetric}. 

\begin{proof}[Proof of Theorem \ref{thisoperimetric}]
Theorem \ref{thisoperimetric} follows from Theorem \ref{iteration}, by taking $u=\chi_E$ to be the 
characteristic function of $E$. 
\end{proof}

\medskip

\subsection*{Acknowledgements} D.P. has been partially supported by the PRIN 2010 MIUR project ``Problemi differenziali di evoluzione: 
approcci deterministici e stocastici e loro interazioni".  Y.S. has been partially supported by 
the ERC grant $\epsilon$ ``Elliptic Pde's and Symmetry of Interfaces and Layers
for Odd Nonlinearities", and the ANR project "HAB".
M.N. and D.P. are members of the italian CNR-GNAMPA. 


\end{document}